\documentclass[12pt]{amsart}
\usepackage{amssymb,amsthm,amsmath,euscript}
\usepackage{fancybox}
\usepackage{listings}
\lstdefinelanguage{magma}
{mathescape}
\newtheorem{thm}{Theorem}
\newtheorem{prop}[thm]{Proposition}
\newtheorem{lem}[thm]{Lemma}

\theoremstyle{remark}
\newtheorem{rem}[thm]{Remark}
\theoremstyle{definition}
\newtheorem{exa}[thm]{Example}
\newtheorem{dfn}[thm]{Definition}

\newcommand{\eD}{\EuScript{D}}
\newcommand{\teD}{\tilde{\EuScript{D}}}
\newcommand{\FF}{\mathbb{F}}

\newcommand{\cC}{\mathcal{C}}
\newcommand{\cG}{\mathcal{G}}

\newcommand{\allone}{\mathbf{1}}
\newcommand{\allzero}{\mathbf{0}}
\newcommand{\tg}[1]{T_{#1}}
\newcommand{\tgc}[1]{C(T_{#1})}
\newcommand{\ttgc}[1]{\hat{C}(T_{#1})}
\DeclareMathOperator{\Aut}{Aut}
\DeclareMathOperator{\wt}{wt}
\DeclareMathOperator{\we}{we}
\DeclareMathOperator{\support}{supp}

\DeclareMathOperator{\Rad}{Rad}
\DeclareMathOperator{\rad}{rad}
\DeclareMathOperator{\GL}{GL}
\DeclareMathOperator{\RM}{RM}
\DeclareMathOperator{\id}{id}
\DeclareMathOperator{\Hom}{Hom}

\title{On triply even binary codes}

\author[K. Betsumiya]{Koichi Betsumiya}
\address{Graduate School of Science and Technology,
Hirosaki University, Hirosaki, 036-8561 Japan}
\email{betsumi@cc.hirosaki-u.ac.jp}

\author[A. Munemasa]{Akihiro Munemasa}
\address{Graduate School of Information Sciences,
Tohoku University, Sendai, 980-8579 Japan}
\email{munemasa@math.is.tohoku.ac.jp}

\subjclass[2000]{05E20, 05E30, 94B05}

\keywords{binary code, divisible code, doubly even code, triangular graph,
framed vertex operator algebra}

\date{August 3, 2011}

\begin{document}

\begin{abstract}
A triply even code is a binary linear code in which the weight of
every codeword is divisible by $8$. We show how two doubly even
codes of lengths $m_1$ and $m_2$ can be combined to make a
triply even code of length $m_1+m_2$, and then prove that
every maximal triply even code of length $48$ can be obtained
by combining two doubly even codes of length $24$ in a certain way.
Using this result, we show that there are exactly $10$ maximal
triply even codes of length $48$ up to equivalence.
\end{abstract}

\maketitle

\section{Introduction}

For the past few decades, extensive research of doubly even binary linear
codes has been done.
These codes turned out to be connected with objects in various areas,
for example, sphere packing problem,
combinatorial designs, finite groups, integral lattices, modular forms and so on
\cite{SPLAG,Rains-Sloane}.
In this paper, we are concerned with a subclass of the class of
doubly even codes, called triply even binary codes.
A triply even code is a binary linear code
in which every codeword has weight divisible by $8$,
in other words, a binary divisible code of level $3$
in the sense of \cite{Liu2006}.
Dong, Griess and H\"ohn \cite{DGH} pointed out that a certain
triply even binary code of length $48$ arose naturally from a Virasoro frame of
the moonshine vertex operator algebra $V^\natural$. 
Subsequently, Miyamoto \cite{M} found a construction method of 
$V^\natural$ from that code.
Lam and Yamauchi \cite{LY} formulated
this construction for the class of framed
vertex operator algebras.
To be precise, a holomorphic framed vertex operator algebra
of central charge $n$ is constructed from
a triply even code of length $2n$ whose dual is even.
Unlike doubly even codes, the classification of
all triply even codes of modest lengths has not been
established yet.

The purpose of this paper is to develop a basic theory of
maximal triply even codes, and to give
a classification of maximal triply even codes of length $48$.
Since any triply even code of length up to $48$ can be
regarded as a subcode of some
maximal triply even codes of length $48$, one can derive
easily the classification of all triply even codes of lengths up to $48$.
It turns out that every maximal triply even code of length
$n$ with $n\equiv0\pmod{8}$ and $n\leq40$
is obtained as the generalized doubling $\teD(C)$
of a maximal doubly even code $C$ (see Definition~\ref{dfn:stdbwr}),
and $n=48$ is the smallest length with $n\equiv0\pmod{8}$
for which there exists
a maximal triply even code not equivalent to $\teD(C)$ for any
doubly even self-dual code $C$.
The unique maximal triply even code $\ttgc{10}$
of length $48$ not equivalent
to generalized doublings is obtained by augmenting the code $\tgc{10}$ of length $45$
generated by the adjacency matrix of the triangular graph $\tg{10}$.

A principal consequence of our classification is that
every structure code of the moonshine vertex operator
algebra lies in the generalized doubling $\teD(C)$ of a 
doubly even self-dual code $C$ of length $24$. This 
follows from \cite[Proof of Lemma~5.5]{DGH}, together with the
fact that $\ttgc{10}^\perp$ has minimum weight $2$.
Moreover, our result constituted a foothold for
further developments toward the classification of
holomorphic framed vertex operator
algebras of central charge $24$. 
The headway of the developments is as follows.
By Lam and Yamauchi \cite{LY}, every triply even code
of length a multiple of $16$ containing the all-ones vector
is the structure code of some holomorphic framed vertex operator algebra.
So it is natural to ask whether there are any new holomorphic
framed vertex operator algebras of central charge $24$
having one of the codes we found as the structure code.
Lam \cite{Lam} recently constructed ten new holomorphic
vertex operator algebras of central charge $24$
using subcodes of $\ttgc{10}$.
Subsequently, Lam and Shimakura \cite{LS} constructed seven new
holomorphic vertex operator algebras of central charge $24$
using subcodes of $\teD(d_{16}^{+}\oplus e_8)$ and $\teD(e_8^{\oplus3})$.
They also obtained 
a complete classification of
the Lie algebra structures of holomorphic framed vertex operator
algebras of central charge $24$.
These holomorphic framed vertex operator algebras
correspond to some of the conformal field theories predicted
to exist by Schellekens \cite{Sch}.

This paper is organized as follows.
In Section~\ref{sec:construction}, properties and some construction
methods of
triply even codes are given.
In Section~\ref{sec:maximality}, we prove that some maximal
triply even codes can be constructed from doubly even self-dual codes
by the doubling process.
In Section~\ref{sec:triangular_graph}, an infinite series of maximal triply even
codes is constructed by triangular graphs and some properties
of the codes in this class are given.
In Section~\ref{sec:gendoubling}, a method for constructing a triply
even code from a pair of doubly even codes is given.
The main result in Section~\ref{sec:gendoubling} states that every
maximal triply even code is obtained from a pair of doubly even codes
containing their radicals. 
In Section~\ref{sec:algorithms}, an efficient method is described 
for determining
whether a given doubly even code contains its radical.
In Section~\ref{sec:classificationb},
we show that the method described in Section~\ref{sec:gendoubling}
gives all maximal triply even codes of length $48$ and, as a result,
a classification of maximal triply
even codes of length $48$ is given.
In Section~\ref{sec:classificationc}, a classification of maximal triply
even codes of lengths $8$, $16$, $24$, $32$ and $40$ is given.
The result is available electronically
from \cite{BD}, as well as
a complete program in {\sc Magma} \cite{MAGMA}
needed to produce it.

\section{Basic constructions for triply even codes}\label{sec:construction}

Throughout the paper, a code will mean a binary linear code,
or equivalently, a linear subspace of the vector space $\FF_2^n$ over the
field $\FF_2$ of two elements.
The support of a vector $u=(u_1,\dots,u_n)\in\FF_2^n$ is the set
$\support(u)=\{i\mid u_i=1\}$,
and the weight of $u$ is $\wt(u)=|\support(u)|$.
A {\em triply even} code is a code
in which every codeword has weight divisible by $8$.
A {\em doubly even} code is a code
in which every codeword has weight divisible by $4$.

In this section, we give basic properties of triply even codes,
and construction methods of triply even codes
from doubly even codes.
An $[n,k]$ code is a code $C\subset\FF_2^n$ with $\dim C=k$,
and $n$ is called the length of $C$.
For codes $C$ and $D$ of length $n$, $C$ is {\em equivalent} to $D$
if $C = D^\sigma$ for some coordinate permutation $\sigma\in S_n$.
The {\em automorphism group} $\Aut(C)$ of $C$ is defined as
$\{\sigma\in S_n\mid C = C^{\sigma}\}$.
The linear span of a subset $S\subset\FF_2^n$ over $\FF_2$
is denoted by $\langle S\rangle$.
For $u,v\in\FF_2^n$, we define $u\ast v$ to be
the vector in $\FF_2^n$ with
$\support(u\ast v)=\support(u)\cap\support(v)$.
For $C, D\subset\FF_2^n$, we define
$C\ast D :=\langle u\ast v\mid u\in C, v\in D\rangle$.
For vectors $u\in\FF_2^m$ and $v\in\FF_2^n$,
we denote by $(u\mid v)\in\FF_2^{m+n}$ the vector
obtained by concatenating $u$ and $v$.
For subsets $C\subset\FF_2^{n_1}$, $D\subset\FF_2^{n_2}$,
we define the {\em direct sum} of $C$ and $D$ as
\[
 C\oplus D =
 \{(u\mid v)\in \FF_2^{n_1+n_2} \mid u\in C,\; v\in D\}.
\]
If $C$ and $C'$ (resp.\ $D$ and $D'$) are codes of length $n_1$ (resp.\ $n_2$)
then $(C\oplus D)\ast(C'\oplus D')=(C\ast C')\oplus(D\ast D')$.
A code $C$ is said to be {\em decomposable} if it is a direct sum of two codes.
We denote by $\allone_n\in\FF_2^n$ and $\allzero_n\in\FF_2^n$,
the all-ones vector, the zero vector, respectively.
We will omit the subscript if there is no confusion.

For vectors $u=(u_1,\dots,u_n),v=(v_1,\dots,v_n)\in\FF_2^n$, we
denote by $u\cdot v$ the standard inner product $\sum_{i=1}^n u_iv_i$.
The dual code of a code $C$ is defined as
$\{u\in\FF_2^n\mid u\cdot v = 0 \text{ for any } v\in C\}$
and is denoted by $C^\perp$.
A code $C$ is {\em self-dual} (resp.\ {\em self-orthogonal})
if $C=C^\perp$ (resp.\ $C\subset C^\perp$).
There exists a doubly even self-dual code of length $n$,
if and only if $n$ is divisible by  $8$.
If $C$ and $D$ are codes, then $(C\oplus D)^\perp = C^\perp \oplus
D^\perp$.

The following lemma is a special case of \cite[Theorem 5.3]{Ward6}
(see also {\cite[Proposition~2.1]{Liu}}).
\begin{lem}\label{lem:gencond}
Let $C=\langle S \rangle$ be a code generated by
a set $S$.
Then $C$ is a triply even code if and only if
the following conditions hold for any $u, v, w\in S$:
\begin{align}
 \wt(u) & \equiv 0 \pmod{8},\label{cnd:gencondi}\\
\wt(u\ast v) & \equiv 0 \pmod{4},\label{cnd:gencondii} \\
\wt(u\ast v \ast w) & \equiv 0 \pmod{2}.\label{cnd:gencondiii}
\end{align}
\end{lem}


\begin{dfn}\label{dfn:rad}
Let $C$ be a
doubly even
code of length $n$.
We define functions
\begin{align*}
Q:  C  &\longrightarrow \FF_2, &
 u&\mapsto \frac{\wt(u)}{4} \bmod{2}, \\
B:  C\times C^\perp & \longrightarrow \FF_2, &
(v, u) &\mapsto \frac{\wt(v\ast u)}{2} \bmod{2}, \\
T:  \FF_2^n\times\FF_2^n\times\FF_2^n  &\longrightarrow \FF_2, &
(u, v, w)&\mapsto \wt(u\ast v\ast w) \bmod{2}.
\end{align*}
Clearly, the following equalities hold:
\begin{align}
Q(x+y)&=Q(x)+Q(y)+B(x,y)\quad
(x,y\in C),\label{eq:QB}\\
B(x,y+z)&=B(x,y)+B(x,z)+T(x,y,z)\quad
(x\in C,\;y,z\in C^\perp),\label{eq:BT1}\\
B(x+y,z)&=B(x,z)+B(y,z)+T(x,y,z)\quad
(x,y\in C,\;z\in C^\perp),\label{eq:BT2}\\
T(x,y,z)&=0\quad
(x,y\in C, z\in(C\ast C)^\perp).\label{eq:Tzero}
\end{align}

The \emph{doubly even radical} $\rad C$,
and the \emph{triply even radical}
$\Rad C$ are defined as
\begin{align*}
\rad C & =\{y\in C^\perp\mid B(x,y)=0\;(\forall x\in C)\}, \\
\Rad C & =\{x\in \rad C \mid Q(x)=0\}.
\end{align*}
Clearly
\begin{align}
\rad(C\oplus D)&=\rad C \oplus \rad D \label{eq:radoplus},\\
\Rad(C\oplus D)&\supset\Rad C \oplus \Rad D \label{eq:Radoplus}
\end{align}
hold.
\end{dfn}

In general, the radicals $\rad C$, $\Rad C$
are not linear and not necessarily contained in $C$,
even if $C$ is triply even.
An example is $C=\langle\allone_8\rangle\oplus\langle\allone_8\rangle$.
However, the following holds.

\begin{lem}\label{lem:R1}
Let $C$ be a doubly even code.
Then $\rad C\subset (C\ast C)^\perp$.
\end{lem}
\begin{proof}
We note that
$(C\ast C)^\perp=\{z\in C^\perp\mid T(x,y,z)=0 \text{ for any }x ,y\in C\}$.
Suppose $x,y\in C$ and $z\in\rad C$.
Since $x+y\in C$,
we have $T(x,y,z)=0$ by (\ref{eq:BT2}).
Thus $z\in(C\ast C)^\perp$, and
the result follows.
\end{proof}

\begin{lem}\label{lem:R0}
Let $C$ be a doubly even code, and suppose $x,y\in\rad C$.
\begin{enumerate}
\item If $y\in C$, then $x+y\in\rad C$.\label{cd:R0_1}
\item If $x+y\in C$, then $x+y\in\rad C$.\label{cd:R0_2}
\end{enumerate}
\end{lem}
\begin{proof}
Observe that, by Lemma~\ref{lem:R1}, $x\in(C\ast C)^\perp$ holds.
For any $z\in C$, we have
$B(z,x+y)=T(x,y,z)$ by (\ref{eq:BT1}).
If $y\in C$, then $T(x,y,z)=0$. Thus \ref{cd:R0_1} holds.
If $x+y\in C$, then 
$T(x,y,z)=T(x,x+y,z)+T(x,x,z)=0$.
Thus \ref{cd:R0_2} holds.
\end{proof}

\begin{lem}\label{lem:RC}
Let $C$ be a doubly even code, and suppose $x\in\rad C$ and $z\in\Rad C$.
Then
\begin{align}
x+C\cap\rad C&=(x+C\cap(C\ast C)^\perp)\cap\rad C,\label{eq:RC1}\\
z+C\cap\Rad C&=(z+C\cap(C\ast C)^\perp)\cap\Rad C.\label{eq:RC2}
\end{align}
\end{lem}
\begin{proof}
The containment
$x+C\cap\rad C\subset (x+C\cap(C\ast C)^\perp)\cap\rad C$
follows from Lemma~\ref{lem:R1} and Lemma~\ref{lem:R0}\ref{cd:R0_1}.
As for the reverse containment,
suppose $y\in C\cap(C\ast C)^\perp$ and $x+y\in\rad C$.
Since $x\in\rad C$,
we have $y\in\rad C$ by Lemma~\ref{lem:R0}\ref{cd:R0_2}.
Thus $x+y\in x+C\cap\rad C$ and (\ref{eq:RC1}) holds.

From (\ref{eq:RC1}),
\[
 (z+C\cap(C\ast C)^\perp)\cap\Rad C = (z+C\cap\rad C)\cap\Rad C.
\]
Suppose $y\in C\cap\rad C$.
Since $\wt(z)\equiv 0\pmod{8}$, $z+y\in\Rad C$ if and only if
$\wt(y)\equiv 0\pmod{8}$. Therefore
\[
 (z+C\cap(C\ast C)^\perp)\cap\Rad C = z+C\cap\Rad C.
\]
Thus \ref{cd:R0_2} holds.
\end{proof}

\begin{lem}\label{lem:RL}
Let $C$ be a doubly even code and $D=(C\ast C)^\perp\cap C$.
Then the restriction $B|_{C\times D}$ of $B$
to $C\times D$ is a bilinear pairing and $Q|_D$ is a quadratic form
with associated bilinear form $B|_{D\times D}$.
Moreover, $C\cap\rad C$ and $C\cap\Rad C$ are linear subcodes of $C$.
In particular, if $\rad C\subset C$ (resp.\ $\Rad C\subset C$), then
 $\rad C$ (resp.\ $\Rad C$) is linear.
\end{lem}
\begin{proof}
First, note that
since $C\subset C\ast C$,
we have $D\subset (C\ast C)^\perp\subset C^\perp$.
For any $x, y\in C$ and $z\in D$, we have $T(x,y,z)=0$ by (\ref{eq:Tzero}),
hence
$B(x+y,z)= B(x,z)+B(y,z)$
by (\ref{eq:BT2}).
Also, for any $x\in C$ and $y, z\in D$, we have $T(x,y,z)=0$ by (\ref{eq:Tzero}),
hence
$B(x,y+z)= B(x,y)+B(x,z)$
by (\ref{eq:BT1}).
Therefore, $B$ is a bilinear pairing on $C\times D$, and
$Q|_D$ is a quadratic form with associated bilinear form
$B|_{D\times D}$ by (\ref{eq:QB}).

By Lemma~\ref{lem:R1},
$C\cap\rad C=\{y\in D\mid B(x,y)=0\;\text{ for any }x\in C\}$.
Since $B|_{C\times D}$ is linear in the second variable,
$C\cap\rad C$ is a linear subcode of $C$.

Also, by (\ref{eq:QB}), $Q$ is linear on $C\cap\rad C$.
Then,
$C\cap\Rad C =\{x\in C\cap\rad C\mid Q(x)=0\}$
is a linear subcode of $C$.
\end{proof}



\begin{dfn}\label{dfn:stdbwr}
Let $C$ be a code of length $n$ and set $R = C\cap \Rad C$.
We define the extended doubling $\eD(C)$ and the generalized
doubling $\teD(C)$ as
\begin{align}\label{eq:eDC}
\eD(C)&=\langle (\allone_n|\allzero_n),(\allzero_n|\allone_n),
\{(x|x)\mid x\in C\}\rangle,\\
\label{eq:teDC}
\teD(C)&=\langle R\oplus \allzero_n,\{(x|x)\mid x\in C\}\rangle.
\end{align}
\end{dfn}
We note that if $C$ is a doubly even code,
then $\teD(C)$ is a triply even code and
\begin{equation}\label{eq:dimteD}
\dim \teD(C)=\dim C + \dim (C\cap\Rad C).
\end{equation}
Note also that if $C$ is a doubly even $[n, d]$ code and
$n\equiv 0 \pmod{8}$,
then $\eD(C)$ is a triply even code of length $2n$,
dimension $d+1$ or $d+2$, depending on $\allone\in C$ or not.
In particular, if $C$ is a doubly even self-dual code of length $n$,
then $\eD(C)$ is a triply even $[2n, n+1]$ code.
This is a particularly important construction in connection with
framed vertex operator algebras and lattices (see \cite{HLM}).
In the next section, we give a sufficient condition for $C$ under
which $\teD(C)$ is a maximal triply even code.

\section{Maximality of triply even codes}\label{sec:maximality}
In this section, we discuss {\em maximal} triply even codes, that is,
triply even codes not contained in any larger triply even code.



\begin{lem}\label{lem:maxcondbyrad}
If $C$ is a triply even code, then $C\subset\Rad C$.
Moreover, equality holds
if and only if $C$ is a maximal triply even code.
\end{lem}
\begin{proof}
The first part is immediate from Lemma~\ref{lem:gencond}.
For a vector $x$, Lemma~\ref{lem:gencond} implies that
$\langle C, x\rangle$ is a triply even code if and only if
$x\in (C\ast C)^\perp\cap \Rad C=\Rad C$ by Lemma~\ref{lem:R1}.
Thus the result follows.
\end{proof}

\begin{lem}\label{lem:indcomdesdii}
Let $C=\bigoplus_{i=1}^kC_i$ be a maximal doubly even code
where $C_i$ is an indecomposable component of length $n_i$ for $i=1,\ldots, k$.
Then
\begin{align}
\rad C = \bigoplus_{i=1}^k \langle\mathbf{s}_{i}\rangle,\label{eq:isdstar} \\
\Rad C = \bigoplus_{i=1}^k \langle\mathbf{t}_{i}\rangle.\label{eq:isdstariix}
\end{align}
where
\begin{align*}
\mathbf{s}_i =\begin{cases}\allone_{n_i} & n_i \equiv 0 \pmod{4} \\
	       \allzero_{n_i} & n_i \not\equiv 0 \pmod{4},\end{cases} \\
\mathbf{t}_i =\begin{cases}\allone_{n_i} & n_i \equiv 0 \pmod{8} \\
	       \allzero_{n_i} &  n_i \not\equiv 0 \pmod{8}.\end{cases}
\end{align*}
In particular, if $C$ is a doubly even self-dual code, then
\begin{align}
\rad C &=\Rad C =\bigoplus_{i=1}^k\langle\allone_{n_i}\rangle\label{eq:sdrad},\\
\teD(C) &\cong \bigoplus_{i=1}^k\teD(C_i)=\bigoplus_{i=1}^k\eD(C_i).
\label{eq:CoD2}
\end{align}
\end{lem}
\begin{proof}
By (\ref{eq:radoplus}),
it suffices to prove (\ref{eq:isdstar}) when
$C$ is indecomposable.
Suppose
$v\in \rad C$ and $x\in C$.
Then
\begin{equation}\label{eq:vstarxde}
 \wt(v\ast x)\equiv 0 \pmod{4}
\end{equation}
By Lemma~\ref{lem:R1}, $v\in(C\ast C)^\perp$.
Then for any $y\in C$, $0=v\cdot(x\ast y)=(v\ast x)\cdot y$.
Hence
\begin{equation}\label{eq:vstarxd}
v\ast x\in C^\perp.
\end{equation}
By (\ref{eq:vstarxde}) and (\ref{eq:vstarxd}),
$\langle C, v\ast x\rangle$
is a doubly even code.
By maximality, $v\ast x\in C$.
Also since $x\in C$ was arbitrary,
$C$ is the direct sum of codes supported by $\support(v)$ and its complement.
Since $C$ is indecomposable, we obtain $v\in\langle\allone\rangle$.
Hence $\rad C \subset \langle\allone\rangle$.
Therefore (\ref{eq:isdstar}) holds.

We claim that
there is at most one $i$ such that $n_i\equiv 4\pmod{8}$.
If there are distinct $i,j$ such that $i\equiv j\equiv 4\pmod{8}$,
then  $C_i\oplus C_j$ is not a maximal doubly even code.
This contradicts maximality of $C$.
Therefore (\ref{eq:isdstariix}) follows from (\ref{eq:isdstar}).

If $C$ is a doubly even self-dual code, then each $C_i$ is a doubly even
self-dual code, hence $n_i$ is divisible by $8$.
Now, (\ref{eq:sdrad}) follows from (\ref{eq:isdstar}) and (\ref{eq:isdstariix}).

By
(\ref{eq:sdrad}), we have
\begin{align*}
C\cap \Rad C&=\bigoplus_{i=1}^k \langle\allone_{n_i}\rangle\\
&=\bigoplus_{i=1}^k C_i\cap\Rad C_i
\end{align*}
and hence $\teD(\bigoplus_{i=1}^k C_i)=
\bigoplus_{i=1}^k \teD(C_i)$. Since
$C_i$ is indecomposable, (\ref{eq:sdrad}) implies
$\teD(C_i)=\eD(C_i)$. This proves
(\ref{eq:CoD2}).
\end{proof}

\begin{prop}\label{prop:teDSD}
For any doubly even self-dual code $C$,
$(\teD(C)\ast\teD(C))^\perp=\teD(C)$.
In particular $\teD(C)$ is a maximal triply even code.
\end{prop}
\begin{proof}
Suppose that $C$ is an indecomposable doubly even self-dual code
of length $2n$.
Then (\ref{eq:CoD2}) implies
$\teD(C)\ast\teD(C)=\eD(C)\ast\eD(C)=C\oplus C + \eD(C\ast C)$,
hence $\dim(\teD(C)\ast\teD(C))=3n-1=4n-\dim\teD(C)$.
This implies that $(\teD(C)\ast\teD(C))^\perp=\teD(C)$.
By (\ref{eq:CoD2}), the identity holds
also for any decomposable double even self-dual code $C$.
Now $\rad \teD(C)\subset \teD(C)$ by Lemma~\ref{lem:R1},
and hence $\teD(C)$ is a maximal triply even code
by Lemma~\ref{lem:maxcondbyrad}.
\end{proof}

\begin{exa}\label{exa:sdoublingtec}
It is known that the $[8, 4, 4]$ Hamming code
$e_8=\eD(\langle\allone_4\rangle^\perp)$ is the unique
doubly even self-dual code of length $8$, up to equivalence.
Also,
$d_{16}^{+} =\eD(\langle\allone_8\rangle^\perp)$ and $e_8\oplus e_8$ are
the only doubly even self-dual codes of length $16$, up to equivalence.
By Proposition~\ref{prop:teDSD},
$\teD(e_8)$, $\teD(d_{16}^{+})$, $\teD(e_8\oplus e_8)$ are maximal triply
even code of dimension $5$, $9$ and $10$ respectively.
In particular $\teD(e_8)=\eD(e_8)$ is the Reed--Muller code $\RM(1,4)$
and $\teD(e_8\oplus e_8)=\RM(1,4)^{\oplus2}$.
\end{exa}



\begin{exa}\label{exa:doublingtec}
It is known \cite{PS} that there are precisely $9$ doubly even
self-dual codes of length $24$.
Two of these $9$ codes are decomposable, and they are
$d_{16}^{+}\oplus e_8$ and $e_8^{\oplus3}$. The remaining
$7$ codes are indecomposable and they are denoted by
$g_{24}, d_{24}^{+}, d_{12}^{2+}, (d_{10}e_7^2)^{+},
d_8^{3+}, d_6^{4+}, d_4^{6+}$.
By Proposition~\ref{prop:teDSD},
$\teD(C)$ is a maximal triply even code for any of the $9$
doubly even self-dual codes $C$.
We note from (\ref{eq:sdrad})
that $\Rad C\subset C$ and $\dim\Rad C$ is the
number of indecomposable components. Thus, for
indecomposable doubly even self-dual codes $C$ of length $24$,
$\dim\teD(C)=13$ holds.
Also,
$\teD(d_{16}^{+}\oplus e_8)=
\eD(\eD(\langle\allone_8\rangle^\perp))\oplus \RM(1,4)$
has dimension $14$, while
$\teD(e_8^{\oplus3})=\RM(1,4)^{\oplus3}$
has dimension $15$.
\end{exa}

\begin{rem}
As shown in Example~\ref{exa:doublingtec},
the dimension of maximal triply even codes varies even
if the length is fixed.
The largest possible dimension of triply even codes,
however, has been determined in \cite{Ward4},  and
the codes achieving the largest dimension have been
determined in \cite{Liu}.
\end{rem}

\section{Triply even codes constructed from triangular graphs}\label{sec:triangular_graph}

Let $n$ be a positive integer with $n\ge4$, and let
$\Omega$ be a set of $n$ elements.
We denote by $\binom{\Omega}{2}$ the set of two-element subsets of $\Omega$.
 The triangular graph
$\tg{n}$ has the set of vertices $\binom{\Omega}{2}$, and
two vertices $\alpha,\beta$ are adjacent whenever
$|\alpha\cap\beta|=1$. It is known \cite{Hub}
that the graph $\tg{n}$ is a strongly
regular graph with parameters
\[
(v,k,\lambda,\mu)=\left(\frac{n(n-1)}{2},2(n-2),n-2,4\right).
\]
Let $A_n$ denote the adjacency matrix of $\tg{n}$.
Then every row of $A_n$ has weight $2(n-2)$, and for any
two distinct rows of $A_n$, the size of the intersection
of their supports is either $n-2$ or $4$.
Let $\tgc{n}$ be the binary code with generator matrix $A_n$.

It is clear that the code $\tgc{n}$ is triply even only if $n \equiv 2
\pmod{4}$. The converse also holds by the following lemma.

\begin{lem}[Haemers, Peeters and van Rijckevorsel {\cite[Subsection~4.1]{HPR}}]\label{lem:HPR}
If $n \equiv 2 \pmod{4}$, the weight enumerator of $\tgc{n}$ is
\[
\we_{\tgc{n}}(x) = \sum_{l=0}^{\lfloor(n-1)/4\rfloor}
\binom{n}{2l} x^{2l(n-2l)}.
\]
In particular, $\tgc{n}$ is a triply even of dimension $n-2$.
\end{lem}

Let $\alpha_i =\{i, n\}\in\binom{\Omega}{2}$, and we denote by $r_i$
the row of $A_n$ indexed by $\alpha_i$ i.e.,
$\{k,l\}\in\support(r_i)$ if and only if $|\alpha_i\cap \{k,l\}|=1$.
Then the following lemma holds.

\begin{lem}[Key, Moori and Rodrigues {\cite[Lemma~3.5]{KMR}}]\label{lem:Tnbase}
If $n$ is even, then $\{r_i \mid  i=1,2,\ldots,n-2\}$ is a basis of $\tgc{n}$.
\end{lem}

We note that
the dimension of $\tgc{n}$ has already been
determined by Tonchev \cite[Lemma~3.6.6]{Ton} and Brouwer and Van Eijl \cite{BvE}.
An explicit basis of $\tgc{n}$ is needed in the sequel
to establish maximality of $\tgc{n}$.
The weight enumerator given in Lemma~\ref{lem:HPR} can also be derived from
the basis.

\begin{lem}\label{lem:TnastTnbase}
If $n$ is even, then $\{r_i\ast r_j \mid 1\le i \le j \le n-2\}$
is a basis of $\tgc{n}\ast \tgc{n}$.
In  particular,
\[
 \dim(\tgc{n}\ast{}\tgc{n}) = \frac{(n-1)(n-2)}{2}.
\]
\end{lem}
\begin{proof}
Observe that, for $1\le i<j<n$, we have
\begin{equation}\label{eq:supprirj}
 \support(r_i\ast r_j) = \{\{i,j\}\}\cup \{\alpha_k \mid 1\le k< n,
\; k\neq i, j\}.
\end{equation}
Suppose
\[
\sum_{i=1}^{n-2} c_ir_i + \sum_{1\le i<j\le n-2} c_{i,j}r_i\ast r_j = 0,
\]
where $c_i, c_{i,j}\in\FF_2$.
Then $c_i=0$ for $i=1,\ldots,n-2$,
because $|\alpha_i\cap\{j,n-1\}|=1$ if and only if $i=j$.
Thus
\[
 \sum_{1\le i<j\le n-2} c_{i,j}r_i\ast r_j = 0.
\]
For $i,j,k,l\in\{1, \ldots, n-2\}$ with $i\neq j$, $k\neq l$,
(\ref{eq:supprirj}) implies
$\{k,l\}\in\support(r_{i}\ast r_{j})$
if and only if $\{k,l\}=\{i,j\}$.
This implies $c_{i,j}= 0$.
\end{proof}

\begin{lem}\label{lem:mTn}
If $n\equiv 2\pmod{4}$, then
$(\tgc{n}\ast \tgc{n})^\perp = \tgc{n} +\langle\allone\rangle$.
In particular, $\tgc{n}$ is a maximal triply even code.
\end{lem}
\begin{proof}
By $(\tgc{n}\ast \tgc{n})^\perp \supset \tgc{n}+\langle\allone\rangle$ and
comparing the dimensions using Lemmas~\ref{lem:Tnbase} and
\ref{lem:TnastTnbase},
we obtain $(\tgc{n}\ast{}\tgc{n})^\perp=\tgc{n}+\langle\allone\rangle$.
Since $\wt(\allone)=\frac{n(n-1)}{2} \equiv 1 \pmod{2}$,
Lemma~\ref{lem:R1} implies
$\Rad \tgc{n}\subset \tgc{n}$.
Thus $\tgc{n}$ is a maximal triply even code by Lemma~\ref{lem:maxcondbyrad}.
\end{proof}

We define $\ttgc{n}$ to be the code of length
$l = 8 \lceil \frac{1}{8}\frac{n(n-1)}{2}\rceil$ constructed from $\tgc{n}$
together with the all-ones vector of length $l$, i.e.,
$\ttgc{n} = \langle\allone_l\rangle + \tgc{n}\oplus \allzero_{l'}$
where $l' = l- \frac{n(n-1)}{2}$.

\begin{thm}
If $n \equiv 2 \pmod{4}$, then
$\ttgc{n}$ is a maximal triply even code.
\end{thm}
\begin{proof}
Let $l=8 \lceil \frac{1}{8}\frac{n(n-1)}{2}\rceil$.
Then
\begin{align*}
(\ttgc{n}\ast{}\ttgc{n})^\perp
&=
(\langle\allone_l\rangle+(\tgc{n}\ast \tgc{n})\oplus\allzero)^\perp
\\ &=
\langle\allone_l\rangle^\perp\cap
((\tgc{n}\ast \tgc{n})^\perp\oplus\FF_2^{l'})
\\ &=
\langle\allone_l\rangle^\perp\cap
((\tgc{n}+\langle\allone\rangle)\oplus\FF_2^{l'}) && \text{(by Lemma~\ref{lem:mTn})}
\\ &=
\tgc{n}\oplus\langle\allone_{l'}\rangle^\perp
+\langle\allone_l\rangle
\\ &=
\ttgc{n}+\allzero\oplus\langle\allone_{l'}\rangle^\perp.
\end{align*}
Since
$l'<8$,
Lemma~\ref{lem:R1} implies
$\Rad \tgc{n}\subset \tgc{n}$.
The result follows from Lemma~\ref{lem:maxcondbyrad}.
\end{proof}

\section{Triply even codes constructed from pairs of doubly even codes
with isometries}\label{sec:gendoubling}

In Section~\ref{sec:construction}, we gave construction methods
for a triply even code from a doubly even code. In this section, we give
a generalization of these construction methods for a pair of doubly even codes.

For a set of coordinates $\{i_1, i_2, \ldots,i_t\}
\subset\{1,2,\ldots,n\}$,
let $\pi:\FF_2^n\rightarrow\FF_2^t$,
$\pi':\FF_2^n\rightarrow\FF_2^{n-t}$
be the projection to the set of coordinates
$\{i_1, i_2, \ldots,i_t\}$,
$\{j_1,\ldots,j_{n-t}\}$, respectively,
where $\{j_1,\ldots,j_{n-t}\}=\{1,\ldots,n\}\setminus\{i_1,\ldots,i_t\}$.
For a code $C$ of length $n$, the {\em punctured code} and the {\em shortened code} of $C$ on
a set of coordinates
$\{i_1, i_2, \ldots, i_t\}$ are the codes
$\pi'(C)$, $\{\pi'(c)\mid c\in C, \; \pi(c)=\allzero\}$, respectively.

Let $C_1$ and $C_2$ be doubly even codes
and $R_i$ be a subcode of $C_i\cap\Rad C_i$ for $i=1,2$.
A bijective linear map
\begin{equation}\label{eq:isometric}
f:C_1/R_1\to C_2/R_2
\end{equation}
is called an {\em isometry} if
$\wt(x_1)\equiv \wt(x_2) \pmod{8}$ for any $x_1+R_1\in C_1/R_1$ and
$x_2+R_2\in f(x_1+R_1)$.
We note that if $x+R_1 = y+R_1$ with $x, y\in C_1$,
then $\wt(x)\equiv\wt(y)\pmod{8}$.
The set of isometries (\ref{eq:isometric}) is denoted by
$\Phi(C_1/R_1, C_2/R_2)$.

For an isometry $f\in\Phi(C_1/R_1, C_2/R_2)$, we define a code
\begin{equation}\label{eq:consta}
D(C_1, C_2, R_1, R_2, f)=\{(x_1|x_2)\mid x_1\in C_1,\;x_2\in f(x_1+R_1)\}.
\end{equation}
Since $f$ is a bijective linear map,
\begin{equation}\label{eq:constb}
D(C_1, C_2, R_1, R_2, f)=\{(x_1|x_2)\mid x_2\in C_2,\;x_1\in f^{-1}(x_2+R_2)\}.
\end{equation}

\begin{prop}\label{prop:D}
Let $C_i$ be a doubly even code of length $m_i$ for $i=1,2$
and $R_i$ be a subcode of $C_i\cap\Rad C_i$.
If $f\in\Phi(C_1/R_1, C_2/R_2)$,
then the code
$D(C_1, C_2, R_1, R_2, f)$
is a triply even code of length $m_1+m_2$
of dimension $\dim C_1 + \dim R_2 = \dim R_1 + \dim C_2$.
\end{prop}
\begin{proof}
Fix $C_1$, $C_2$, $R_1$, $R_2$ and $f$.
We abbreviate $D(C_1, C_2, R_1, R_2, f)$ as $D$.
Since $f$ is a linear map, $D$ is linear.
Since $C_1$ and $C_2$ are doubly even codes and $f$ is an isometry,
all the weights of elements of $D$ are multiple of $8$, that is,
$D$ is a triply even code.
Moreover
\[
 |D(C_1, C_2, R_1, R_2, f)|=|C_1|\times |f(R_1)|=|C_1|\times|R_2|.
\]
Therefore $\dim D=\dim C_1 + \dim R_2 = \dim R_1 + \dim C_2$.
\end{proof}

Remark that the construction method in Proposition~\ref{prop:D} contains the
constructions $\eD(C)$  in (\ref{eq:eDC}) and  $\teD(C)$ in (\ref{eq:teDC})
as special cases.
Indeed, let $C$ be a doubly even code of length $n$.
Then we have
\begin{align}\label{eq:tdoubling}
\teD(C)&=D(C, C, C\cap\Rad C, C\cap\Rad C, \id)
\end{align}
and if, moreover, $n\equiv 0\pmod{8}$, then
\[
\eD(C)=
D(C+\langle\allone\rangle,C+\langle\allone\rangle,
\langle\allone\rangle,\langle\allone\rangle, \id).
\]

Note that, given doubly even codes $C_1,C_2$ and
subcodes $R_1\subset C_1\cap\Rad C_1$,
$R_2\subset C_2\cap\Rad C_2$,
the set $\Phi(C_1/R_1, C_2/R_2)$ may be empty,
and in this case Proposition~\ref{prop:D} produces no
triply even codes.
We shall give a
necessary and sufficient condition for the set $\Phi(C_1/R_1, C_2/R_2)$
to be non-empty
in Proposition~\ref{prop:singpts} below. First we need
to introduce some terminology.

\begin{dfn}\label{dfn:check}
Let $C$ be a doubly even code, and let $R$ be a subcode
of $C\cap\Rad C$.
Let
\[
X =\{x+R\in C/R\mid \wt(x)\equiv0\pmod{8}\}.
\]
We call the elements of the set $X$
{\em singular points} of $C/R$.
We denote by $\cG_1(C,R)$ the setwise stabilizer
of $X$ in $\GL(C/R)$.
The {\em triply even check code} $\cC(C,R)$
of $(C,R)$ is defined as
\[
\cC(C,R)=\{c=(c_x\in\FF_2\mid x\in X)\in\FF_2^X\mid \sum_{x\in X} c_xx\in R\}.
\]
\end{dfn}
By the definition, $\cG_1(C,R)$ acts on $\cC(C,R)$ as
automorphisms, but the action is not necessarily faithful.
Indeed, $X$ may not span $C/R$.

\begin{prop}\label{prop:singpts}
Let $C_i$ be a doubly even code for $i=1,2$
and $R_i$ be a subcode of $C_i\cap\Rad C_i$.
Suppose that $\dim C_1/R_1 = \dim C_2/R_2$.
Then $\cC(C_1,R_1)\cong\cC(C_2,R_2)$
if and only if
$\Phi(C_1/R_1, C_2/R_2)\neq\emptyset$.
\end{prop}
\begin{proof}
If $\cC(C_1,R_1)\cong\cC(C_2,R_2)$, then there exists a bijection
$f$ from the set $X_1$ of singular points of $C_1/R_1$ to the set
$X_2$ of those of $C_2/R_2$
which induces an equivalence from $\cC(C_1,R_1)$ to $\cC(C_2,R_2)$.
It follows from the definition of the triply even check code that
the bijection $f$ extends to a linear mapping $\langle X_1\rangle\to
\langle X_2\rangle$. Extending further to $C_1/R_1$ in an arbitrary
manner, we obtain an isometry from $C_1/R_1$ to $C_2/R_2$.
The proof of the converse is immediate.
\end{proof}

The next proposition shows that every triply even code
can be constructed by means of the construction described
in Proposition~\ref{prop:D}.

\begin{prop}\label{prop:3da}
Let $D$ be a triply even code of length $n$.
Fix a codeword $x\in D$ of weight $m_1$ with $0<m_1<n$.
Let $S_1=\support(\allone + x)$ and $S_2=\support(x)$ and
let $C_i$ and $R_i$ be the punctured code and the shortened code of $D$
on $S_i$, respectively, for $i=1,2$.
Then
$C_i$ is doubly even, $R_i\subset C_i\cap\Rad C_i$
for $i=1,2$,
and
\[
 D\cong D(C_1, C_2, R_1, R_2, f)
\]
for some $f\in \Phi(C_1/R_1, C_2/R_2)$.

Moreover, if $D$ is maximal, then $\Rad C_i=R_i$ for $i=1,2$.
\end{prop}
\begin{proof}
All the statement except on the last one follows easily from Lemma~\ref{lem:gencond}.
Let $\pi_1 : \FF_2^n\rightarrow\FF_2^{m_1}$, $\pi_2 :
\FF_2^n\rightarrow\FF_2^{n-m_1}$
be the projection to the set of coordinates $\support(x)$,
$\support(\allone+x)$, respectively. Define $\pi:\FF_2^n\rightarrow\FF_2^n$
by $\pi(x) = (\pi_1(x)|\pi_2(x))$ \ ($x\in\FF_2^n$).
Then $C_i=\pi_i(D)$ \ ($i=1,2$) and $D$ is equivalent to $\pi(D)$.
It is clear that the mapping
\begin{align*}
f :  C_1/R_1  &\longrightarrow C_2/R_2 \\
c_1 + R_1 &\mapsto \{ x \in C_2\ \mid (c_1 |  x)\in \pi(D)\}
\end{align*}
is a well-defined isometry and $\pi(D)=D(C_1, C_2, R_1, R_2, f)$.

If $D$ is a maximal triply even code, then so is $\pi(D)$.
This implies $(r_1|\allzero)$, $(\allzero|r_2)\in\Rad \pi(D)$
for $r_1\in \Rad C_1$ and $r_2\in \Rad C_2$.
By Lemma~\ref{lem:maxcondbyrad}, $\Rad \pi(D)\subset \pi(D)$.
Therefore $R_i\subset\Rad C_i$ for $i=1,2$.
Hence the result follows.
\end{proof}



Proposition~\ref{prop:3da} indicates that every triply even code of
length $n$ containing a codeword of weight $m_1$ can be constructed from
a pair of doubly even codes of lengths $m_1$ and $n-m_1$.
We will classify maximal triply even codes of length $48$
by setting $m_1=24$
in Section~\ref{sec:classificationb}.

For fixed codes $C_1$, $C_2$ and
$R_1\subset C_1\cap\Rad C_1$,
$R_2\subset C_2\cap\Rad C_2$
the resulting code
\[
D(C_1, C_2, R_1, R_2, f)
\]
depends on the choice of the isometry $f$.
However, some of these codes are equivalent to each other.
The first algorithm is to check this, that is, we will give a sufficient condition
for two resulting codes to be equivalent.
We need this algorithm to reduce the amount of calculation to be reasonable.

First, we define some groups.
For a code $C$ and a subcode $R\subset C\cap\Rad C$,
we denote by  $\cG_0(C,R)$ the subgroup of $\GL(C/R)$
induced by the action of $\Aut(C)\cap\Aut(R)$ on $C/R$ and
denote by $\cG_1(C,R)$ the subgroup $\Phi(C/R, C/R)$ of $\GL(C/R)$.
By the definition, the group $\cG_0(C,R)$ is a subgroup of $\cG_1(C,R)$.
If $R=C\cap\Rad C$, then we abbreviate $\cG_0(C,R)$,
$\cG_1(C,R)$ as $\cG_0(C)$, $\cG_1(C)$, respectively.
If $f\in\Phi(C_1/R_1, C_2/R_2)$,
then
\begin{align}
\cG_1(C_1, R_1) &= f^{-1}\circ \cG_1(C_2, R_2)\circ f \label{eq:ordereq1}\\
\intertext{and}
\Phi(C_1/R_1, C_2/R_2) &=f\circ \cG_1(C_1, R_1) = \cG_1(C_2, R_2)\circ f. \label{eq:ordereq2}
\end{align}
If we replace $f$ by $\sigma_2\circ f\circ \sigma_1$,
where $\sigma_i\in\cG_0(C_i,R_i)$, then the resulting codes
are equivalent, that is,
\[
D(C_1, C_2, R_1, R_2, f)\cong D(C_1, C_2, R_1, R_2, \sigma_2\circ f\circ \sigma_1).
\]
This means that,
in order to enumerate
\[
\{D(C_1, C_2, R_1, R_2, h)\mid h\in\Phi(C_1/R_1, C_2/R_2)\}
\]
up to equivalence,
we first fix $f\in\Phi(C_1/R_1, C_2/R_2)$, and
it suffices to enumerate the codes
$D(C_1, C_2, R_1, R_2, f \circ g)$
where $g$ runs through a set of representatives for the double cosets
\[
(f^{-1}\circ\cG_0(C_2, R_2)\circ f)\backslash\cG_1(C_1, R_1)/\cG_0(C_1, R_1).
\]

\section{Doubly even codes containing their radicals}\label{sec:algorithms}
In view of Propsition~\ref{prop:3da},
it will be necessary to extract only
those doubly even codes $C$ which
satisfy $\Rad C\subset C$, in order
to enumerate maximal triply even codes.
In this section,
we will give a criteria to verify
whether a doubly even code $C$ contains
its triply even radical i.e., $\Rad C\subset C$.

Throughout this section,
let $C$ be a doubly even code containing $\allone$,
and we denote $(C\ast C)^\perp\cap C$ by $D$.
For $x\in C^\perp$,
one can define a mapping $B_x:C\to\FF_2$ by
$B_x(c)=B(c,x)$ ($c\in C$). By (\ref{eq:BT2}), $B_x$ is linear
when $x\in (C\ast C)^\perp$. Thus we obtain a map
\begin{align*}
\phi:(C\ast C)^\perp &\to \Hom(C, \FF_2) \\
 x &\mapsto B_x.
\end{align*}
By Lemma~\ref{lem:R1}, we can write
\begin{equation}\label{eq:phirad}
\phi^{-1}(0)=\rad C.
\end{equation}
We remark that the map $\phi$ is not linear in general.
More precisely, if we define a bilinear map $\delta$ as
\begin{align*}
\delta:\FF_2^n\times\FF_2^n &\to \Hom(C, \FF_2) \\
 (x,y) &\mapsto (v\mapsto T(x,y,v)),
\end{align*}
then for $x,y\in(C\ast C)^\perp$,
\begin{equation*}
\phi(x+y)=\phi(x)+\phi(y)+\delta(x,y)
\end{equation*}
holds by (\ref{eq:BT1}).
In particular, (\ref{eq:Tzero}) implies
\begin{equation}
\phi(x+y)=\phi(x)+\phi(y)
\quad(x\in(C\ast C)^\perp, \; y\in D),\label{eq:philn}
\end{equation}
and $\phi$ is linear on $D$.

\newcommand{\D}{D}
\newcommand{\bDb}{D}

The function $Q$ from Definition~\ref{dfn:rad}
can also be defined on $\rad C$,
so we denote it by the same $Q$ as follows.
\begin{equation*}\label{eq:newQ}
Q : \rad C\rightarrow \FF_2,\quad u\mapsto\frac{\wt(u)}{4}\bmod{2}.
\end{equation*}
Then $\Rad C=Q^{-1}(0)$, and
\begin{equation}\label{eq:qlin}
Q(x+y)=Q(x)+Q(y)\quad (x\in C\cap\rad C, \; y\in\rad C).
\end{equation}
\begin{lem}\label{lem:outrad}
For a coset $M\in(C\ast C)^\perp/D$,
the following are equivalent.
\begin{enumerate}
\item $\phi(M)\cap\phi(\D)\neq\emptyset$,\label{cnd:memrad1}
\item $\phi(M)=\phi(\D)$,\label{cnd:memrad1a}
\item $M\cap\rad C\ne\emptyset$.\label{cnd:memrad2}
\end{enumerate}
Moreover, if $C\cap\rad C\neq C\cap\Rad C$,
then each of \ref{cnd:memrad1}--\ref{cnd:memrad2} is equivalent to
\begin{enumerate}
\setcounter{enumi}{3}
\item $M\cap\Rad C\neq \emptyset$.\label{cnd:memrad3}
\end{enumerate}
\end{lem}
\begin{proof}
Equivalence of \ref{cnd:memrad1}--\ref{cnd:memrad2} follows
immediately from (\ref{eq:phirad}) and (\ref{eq:philn}).
Suppose $C\cap \rad C\neq C\cap\Rad C$.
It suffices to show that \ref{cnd:memrad2} implies \ref{cnd:memrad3}.

Suppose $x\in M\cap\rad C$.
If $Q(x)=0$, then clearly \ref{cnd:memrad3} holds,
so suppose $Q(x)=1$.
By assumption, there exists
$y\in C\cap \rad C$ such that $Q(y)=1$.
Then $x+y\in M$,
$\phi(x+y)=\phi(x)+\phi(y)=0$ by (\ref{eq:philn}),
hence $x+y\in\rad C$.
Moreover, $Q(x+y)=Q(x)+Q(y)=0$ by (\ref{eq:qlin}).
Thus $x+y\in\Rad C$, and hence \ref{cnd:memrad3} holds.
\end{proof}

\begin{prop}\label{prop:RadCneq}
Let $C$ be a doubly even code of length a multiple of eight,
containing $\allone$.
Suppose $C\cap \rad C\neq C\cap\Rad C$.
Then $\Rad C\not\subset C$ if and only if
there exists a coset $M\in(C\ast C)^\perp/D$ satisfying 
$\phi(M)\cap\phi(D)\neq\emptyset$ and $M\neq D$.
\end{prop}
\begin{proof}
Since $\Rad C\subset\rad C\subset(C\ast C)^\perp$
by Lemma~\ref{lem:R1},
$\Rad C\not\subset C$
if and only if
$M\cap\Rad C\neq\emptyset$
for some coset $M\in(C\ast C)^\perp/\bDb$
different from $\D$.
The resut then follows from Lemma~\ref{lem:outrad}.
\end{proof}

In view of equivalence of (i) and (ii) in Lemma~\ref{lem:outrad},
one can check the condition $\phi(M)\cap\phi(D)\neq\emptyset$
by testing whether an arbitrarily chosen
element $x\in M$ satisfies $\phi(x)\in\phi(D)$.
Thus, the above proposition gives a convenient criterion for
$\Rad C\subset C$ in terms of coset representatives for
$(C\ast C)^\perp/\bDb$,
provided $C\cap\rad C\neq C\cap\Rad C$.
In the case where $C\cap\rad C=C\cap\Rad C$,
the situation is slightly more complicated.

\begin{lem}\label{lem:outRad}
Suppose $C\cap\rad C=C\cap\Rad C$,
and $M\in(C\ast C)^\perp/D$.
If $M\cap\Rad C\neq\emptyset$, then
\begin{equation}\label{eq:outRad}
M\cap\rad C=M\cap\Rad C.
\end{equation}
\end{lem}
\begin{proof}
By assumption, there exists $x\in M\cap\Rad C$.
Then by Lemma~\ref{lem:RC},
we have
\begin{equation}
M\cap\rad C = x + C\cap\Rad C.\label{eq:rR}
\end{equation}
Since $x\in\Rad C$, (\ref{eq:RC2}) implies
$x+C\cap\Rad C\subset\Rad C$, hence
$M\cap\rad C\subset\Rad C$ by (\ref{eq:rR}).
This proves $M\cap\rad C\subset M\cap\Rad C$,
and the reverse containment is trivial.
\end{proof}
%
%
%
\begin{prop}\label{prop:RadCeq}
Let $C$ be a doubly even code
of length a multiple of eight, containing $\allone$.
Suppose $C\cap\rad C=C\cap\Rad C$.
Let $\{x_1,\dots,x_t\}\subset (C\ast C)^\perp$ be a set of
coset representatives for the cosets $M\in(C\ast C)^\perp/D$
satisfying $\phi(M)\cap\phi(D)\neq\emptyset$ and $M\neq D$.
For each $i\in\{1,\dots,t\}$, choose $y_i\in D$ in such a way
that $\phi(x_i)=\phi(y_i)$.
Then the following are equivalent.
\begin{enumerate}
\item $\Rad C\not\subset C$,\label{cd:RadCeq1}
\item $\wt(x_i+y_i)\equiv0\pmod8$ for some $i\in\{1,\dots,t\}$.
\label{cd:RadCeq2}
\end{enumerate}
\end{prop}
\begin{proof}
First, we note that $\phi(M)\cap\phi(D)\neq\emptyset$ implies
that $\phi(M)=\phi(D)$ by Lemma~\ref{lem:outrad}. Thus there
exists $y_i\in D$ such that $\phi(x_i)=\phi(y_i)$, no matter
how we choose a representative $x_i$ for the coset $x_i+D$.

Suppose \ref{cd:RadCeq1} holds. Take $x\in\Rad C\setminus C$
and set $M=x+D$.
Then $x\in M\cap\Rad C$, and hence (\ref{eq:outRad}) holds
by Lemma~\ref{lem:outRad}.
Also, as $x\notin D$ and $\phi(x)=0$ by (\ref{eq:phirad}), 
$M=x_i+D$ holds for some $i\in\{1,\dots,t\}$.
Thus $x_i+y_i\in M$, while 
$\phi(x_i+y_i)=\phi(x_i)+\phi(y_i)=0$ by
(\ref{eq:philn}). Therefore,
$x_i+y_i\in M\cap\rad C\subset\Rad C$ by
(\ref{eq:outRad}).
This implies $\wt(x_i+y_i)\equiv0\pmod8$.

Conversely, if \ref{cd:RadCeq2} holds,
then $x_i+y_i\in\Rad C\setminus C$, and hence \ref{cd:RadCeq1} holds.
\end{proof}

\section{Classification of maximal triply even codes of length $48$}\label{sec:classificationb}

In this section, we aim to give a classification of maximal triply even codes of length $48$.
In Section~\ref{sec:maximality} and Section~\ref{sec:triangular_graph},
we gave $10$ distinct maximal triply even codes of length $48$.
Now we show that the list is complete for a classification up to equivalence
applying Proposition~\ref{prop:3da} and \ref{prop:D} for $n=48$ and $m_1=m_2=24$.
To do this, we first need to establish the existence of a
codeword of weight $24$ in any maximal triply even code of length $48$.

\begin{lem}\label{lem:24a}
Let $D$ be a maximal triply even code of length $n$.
Let $\Gamma$ be the graph with vertex set $\{1,\dots,n\}$
and edge set
\[
\{\support(x)\mid x\in D^\perp,\;\wt(x)=2\}.
\]
Then the following hold:
\begin{enumerate}
\item every connected component of $\Gamma$ is a complete
graph with at most $8$ vertices,
\item
if there is a connected component of $\Gamma$ with more than
$4$ vertices, then any other connected component has at most
$3$ vertices.
\end{enumerate}
\end{lem}
\begin{proof}
Since $D^\perp$ is a linear code, it is clear that
every connected component of $\Gamma$ is a complete
graph.
Suppose that there is a connected component $K$ of $\Gamma$
with $|K|>8$. Then there exists a vector $x\in\FF_2^n$ with
$\wt(x)=8$ and $\support(x)\subset K$. Since
the restriction of $y$ to $K$ is $\allzero$ or $\allone$ for any $y\in D$,
we have $\wt(x*y)=0$ or $8$. This implies that $\langle D,x\rangle$
is triply even.
Taking $i\in\support(x)$ and $j\in K\setminus\support(x)$, the
vector with support $\{i,j\}$ belongs to $D^\perp$ and is not
orthogonal to $x$. Thus $x\notin D$. This contradicts
the fact that $D$ is maximal, and the proof of (i) is complete.

To prove (ii),
suppose that there are distinct
connected components $K,K'$ of $\Gamma$
with $|K|>4$ and $|K'|\geq4$.
Then there exists a vector $x\in\FF_2^n$ with
$\wt(x)=8$, $|\support(x)\cap K|=|\support(x)\cap K'|=4$. Since
the restriction of $y$ to $K$ or $K'$ is $\allzero$ or $\allone$ for any $y\in D$,
we have $\wt(x*y)=0,4$ or $8$. This implies that $\langle D,x\rangle$
is triply even.
The rest of the proof is exactly the same as (i).
\end{proof}

\begin{lem}\label{lem:24}
Let $D$ be a maximal triply even code of length $48$
containing $\allone$.
Then $D$ has at least one codeword of weight $24$.
\end{lem}
\begin{proof}
By Lemma~\ref{lem:24a}, the number of codewords of $D^\perp$
with weight $2$ is of the form
\[
\sum_{K}\binom{|K|}{2},
\]
where the summation is taken over the set of connected components
of the graph $\Gamma$ defined in Lemma~\ref{lem:24a}. Let
$\lambda_1\geq\lambda_2\geq\cdots$ be the partition of $48$
associated with the decomposition of the vertex set of $\Gamma$
into connected components. Lemma~\ref{lem:24a} implies that
one of the following holds:
\begin{enumerate}
\item $4<\lambda_1\leq 8$ and $\lambda_i\leq3$ for all $i\geq2$,
\item $\lambda_i\leq 4$ for all $i\geq1$.
\end{enumerate}
It is not difficult to show that the maximum value of
$\sum_i\binom{\lambda_i}{2}$ is $\binom{8}{2}+13\binom{3}{2}=
67$ for the case (i), and $12\binom{4}{2}=72$ for the case (ii).
Therefore, we conclude that $D^\perp$ has at most $72$
codewords of weight $2$.

Now suppose that $D$ has no codeword of weight $24$, so that
its weight enumerator is
\[
X^{48}+aX^{40}Y^8+(2^{k-1}-(1+a))(X^{32}Y^{16}+X^{16}Y^{32})
+aX^8Y^{40}+Y^{48},
\]
where $k=\dim D$.
It follows from the MacWilliams identities that the
number of codewords of weight $2$ in $D^\perp$ is
\[
3\cdot2^{8-k}a+ 104 + 2^{11-k}
\]
which is certainly greater than $72$. This is a contradiction.
\end{proof}

In order to construct all maximal triply even codes of length $48$
by means of Proposition~\ref{prop:D} and \ref{prop:3da} for $n=48$ and $m_1=m_2=24$,
it suffices to consider the codes of length $24$ satisfying
$\Rad C_i\subset C_i$ as candidates for $C_1$ and $C_2$.
This is because, if a resulting code $D(C_1, C_2, R_1, R_2, f)$
is maximal, then
$R_i=\Rad C_i$ for $i=1,2$ as we mentioned in Proposition~\ref{prop:3da}.


We are now ready to describe our enumeration
using {\sc Magma} system \cite{MAGMA}.

As the first step, we enumerate all doubly even codes
of length $24$ containing its triply even radical.
Since there is a database of doubly even codes \cite{Miller},
we could make use of it and extract only those which contain
the triply even radical.
However, since every doubly even code is equivalent to a subcode
of the nine doubly even self-dual codes of length $24$ \cite{PS},
we can find all the desired doubly even codes by successively
taking subcodes of codimension one starting from the doubly
even self-dual codes. This approach has an advantage that once
we encounter a doubly even code $C$ with $\Rad C\not\subset C$,
then $\Rad C'\not\subset C'$ for any subcode $C'$ of $C$, so
that it is no longer necessary to consider subcodes of $C$ by Lemma~\ref{lem:maxcondbyrad}.
Table~\ref{tab:doublyeven} gives the numbers of doubly even codes
of length $24$ containing its triply even radical
with each given dimension and dimension of its triply even radical.

\begin{table}[thb]
\caption{The numbers of doubly even code $C$ of length $24$ with $\Rad C\subset C$}
\label{tab:doublyeven}
\begin{center}
\begin{tabular}{|c|c|c|c|c|c|c|}
\hline
$\dim C\setminus\dim\Rad C$ & 1 &2&3&4&5&6\\
\hline
12&  7& 1& 1& 0&0&0\\ \hline
11& 33& 6& 3& 0&0&0\\ \hline
10&130&19&10& 1&0&0\\ \hline
 9&308&40&23& 5&0&1\\ \hline
 8&363&37&25&10&1&1\\ \hline
 7&180&16&10&11&2&1\\ \hline
 6& 27& 2& 0& 4&2&1\\ \hline
 5&  0& 0& 0& 0&1&0\\ \hline

\end{tabular}
\end{center}
\end{table}

As the second step, we enumerate all resulting codes
\[
D(C_1, C_2, \Rad C_1, \Rad C_2, f\circ g)
\]
obtained from the all combinations of doubly even codes $C_1$, $C_2$ above
and a representative $g\in\bar{g}$
for each double coset
\[
\bar{g}\in(f^{-1}\circ\cG_0(C_2, R_2)\circ f)\backslash\cG_1(C_1, R_1)/\cG_0(C_1, R_1),
\]
where $f$ is a fixed element of $\Phi(C_1/\Rad C_1, C_2/\Rad C_2)$
by the procedure given in Proposition~\ref{prop:D}


We denote the set of doubly even codes of length $24$ by
\[
\Delta = \{g_{24}, d_{24}^{+}, d_{12}^{2+}, (d_{10}e_7^2)^{+},
d_8^{3+}, d_6^{4+}, d_4^{6+}, d_{16}^{+}\oplus e_8, e_8^{\oplus3}\},
\]
in accordance with the notation of \cite{PS}.
From the combinations with $C_1=C_2$,
we obtain $1482$ triply even codes.
However, many codes of them of the form (\ref{eq:tdoubling})
turn out not to be maximal. This is because,
if there is a doubly even code $C'$ such that $C\subsetneq C'$
and $\Rad C=\Rad C'$, then
$\teD(C)\subsetneq\teD(C')$. Therefore,
we find that only $216$ codes among the $1482$ codes are possibly maximal.
Then we use Lemma~\ref{lem:maxcondbyrad} to
check maximality, and we are able to confirm that only $30$ codes
among them are maximal.
Each of these $30$ codes turns out to be equivalent to
$\teD(C)$ for some $C\in\Delta$.


From the combinations with $C_1\not\cong C_2$,
we obtain $225$ triply even codes,
and $5$ codes among them are maximal.
One code is equivalent to $\ttgc{10}$.
The other codes are equivalent to a member of
$\{\teD(C) \mid C\in\Delta\}$.
Therefore we obtain the following theorem.
\begin{thm}\label{thm:main}
Every maximal triply even code of length $48$
is equivalent to
$\teD(C)$ for some $C\in\Delta$
or $\ttgc{10}$.
\end{thm}

The computer calculation needed to establish Theorem~\ref{thm:main}
was done by magma \cite{MAGMA}
under the environment using
Intel{\tiny \textregistered} 
Core{\tiny \texttrademark} 2 Duo CPU T7500 @ 2.20GHz, and it took
650.240 seconds.

\section{Classification of maximal triply even codes of lengths $8$,
 $16$, $24$, $32$ and $40$}\label{sec:classificationc}

In this section, we give a classification of maximal triply even codes
of lengths $8$, $16$, $24$, $32$ and $40$ by using a shortening process
from the results of maximal triply even codes of length $48$ in the previous sections.

It is easy to see that every maximal triply even code of length $n$
is a shortened code
of a maximal triply even code of length $n+1$.
From the list of maximal triply even codes of length $48$,
we can derive the list of all maximal triply even codes of shorter lengths
by the shortening process.
The shortened code of $\teD(C)$ on one coordinate has an odd length,
so it cannot be of the form $\teD(C')$ for any $C'$. However,
for lengths divisible by $8$,  the following holds.

\begin{thm}\label{thm:classshorter}
For $n=4, 8, 12, 16$ and\/ $20$,
every maximal triply even code of length $2n$ is of the form $\teD(C)$
for some maximal doubly even code $C$ of length $n$.
\end{thm}

Table~\ref{tab:eight} gives the numbers of
the maximal triply even codes of lengths $8, 16, 24, 32$ and $40$,
up to equivalence.

In Table~\ref{tab:eight}, the first and fifth columns indicate the length of
each doubly even code and each triply even code, respectively.
The second and sixth columns indicate the dimension as well.
The third column indicates the number of indecomposable components
of the doubly even code.
The fourth and seventh columns indicate the number of codes satisfying the condition.
The eighth column gives the other construction method to obtain it.

Note that if $C$ is some maximal doubly even code
and $k$ is the number of self-dual
indecomposable components of $C$, then
$\dim\teD(C)=\dim C+k$ by (\ref{eq:dimteD}) and (\ref{eq:isdstariix}).
For example, there is a unique doubly even $[20,9]$ code $C$ which is
the direct sum of three indecomposable codes, two of which are
self-dual. Then $\teD(C)$ is a triply even $[40,11]$ code. Similarly,
there is a unique doubly even $[24,12]$ code $C$ which is
the direct sum of three indecomposable self-dual codes.
Then $\teD(C)$ is a triply even $[48,15]$ code.

\begin{table}[thb]
\caption{The numbers of maximal triply even codes of lengths multiple of $8$ up to $48$}
\label{tab:eight}
\begin{center}
  \begin{tabular}{|c|c|c|c||c|c|c|l|}
    \hline
     \multicolumn{4}{|l||}{maximal doubly even codes} & \multicolumn{4}{l|}{maximal triply even codes} \\
    \hline
     len & dim & \#compos & \#codes & len & dim & \#codes & remark \\
     \hline
      4  &  1 & 1 & 1 &        8 &  1 & 1 &  \\
     \hline
      8  &  4 & 1 & 1 &       16 &  5 & 1 & $\ttgc{6}$\\
     \hline
     12  &  5 & 1 & 1 &       24 &  5 & 1 & \\
         &    & 2 & 1 &          &  6 & 1 & \\
     \hline
     16  &  8 & 1 & 1 &       32 &  9 & 1 & \\
         &    & 2 & 1 &          & 10 & 1 & \\
     \hline
     20  &  9 & 1 & 7 &       40 &  9 & 7 & \\
         &    & 2 & 2 &          & 10 & 2 & \\
         &    & 3 & 1 &          & 11 & 1 & \\
     \hline
     24  & 12 & 1 & 7 &       48 & 13 & 7 & \\
         &    & 2 & 1 &          & 14 & 1 & \\
         &    & 3 & 1 &          & 15 & 1 & \\
     \hline
         &    &   &   &       48 &  9 & 1 & $\ttgc{10}$ \\
     \hline
  \end{tabular}
\end{center}
\end{table}

\subsection*{Acknowledgments}

The authors would like to thank
Ching Hung Lam, Masaaki Harada and Hiroki Shimakura
for helpful discussions, and Jon-Lark Kim for
bringing the paper \cite{Liu} to the authors' attention.

\newpage
\appendix
/* 
%
\section{A Magma program for classification}

\subsection*{Enumeration of doubly even codes of length 24}
This appendix gives {\sc Magma} scripts to verify Theorem~\ref{thm:main}
in Section~\ref{sec:classificationb}.

It is known that
there are precisely $9$ doubly even self-dual codes
of length $24$ up to equivalence \cite{PS}.
The object \verb+desd24genmats+ is the list of generator matrices in the hexadecimal expression.
Also the object \verb+desd24+ is the list of the codes
\[
\Delta = \{g_{24}, d_{24}^{+}, d_{12}^{2+}, (d_{10}e_7^2)^{+},
d_8^{3+}, d_6^{4+}, d_4^{6+}, d_{16}^{+}\oplus e_8, e_8^{\oplus3}\}.
\]
\begin{lstlisting}[firstline=2,name=prog,firstnumber=1]
******/
desd24genmats:=[
[ 0xC75001, 0x49F002, 0xD4B004, 0x6E3008, 0x9B3010, 0xB66020,
  0xECC040, 0x1ED080, 0x3DA100, 0x7B4200, 0xB1D400, 0xE3A800 ],
[ 0x7FE801, 0x802802, 0x804804, 0x808808, 0x810810, 0x820820,
  0x840840, 0x880880, 0x900900, 0xA00A00, 0xC00C00, 0xFFF000 ],
[ 0x7E0F81, 0xFC0082, 0xFC0104, 0xFC0208, 0xFC0410, 0xFC0820,
  0x820FC0, 0x861000, 0x8A2000, 0x924000, 0xA28000, 0xC30000 ],
[ 0xD003C1, 0xD1A042, 0xD1A084, 0xD1A108, 0xD1A210, 0x01A3E0,
  0x00E400, 0x01C800, 0x017000, 0x720000, 0xE40000, 0xB80000 ],
[ 0x7800E1, 0x88F022, 0x88F044, 0x88F088, 0xF0F0F0, 0x78E100,
  0x78D200, 0x78B400, 0x787800, 0x990000, 0xAA0000, 0xCC0000 ],
[ 0xE24031, 0x738012, 0x738024, 0x91C038, 0x938C40, 0xE1C480,
  0xE1C900, 0x724E00, 0x02D000, 0x036000, 0xB40000, 0xD80000 ],
[ 0xCC6009, 0x66A00A, 0xAAC00C, 0xC6C090, 0x6A60A0, 0xACA0C0,
  0x6CC900, 0xA66A00, 0xCAAC00, 0x00F000, 0x0F0000, 0xF00000 ],
[ 0x0000B1, 0x0000E2, 0x000074, 0x0000D8, 0x7E8100, 0x828200,
  0x848400, 0x888800, 0x909000, 0xA0A000, 0xC0C000, 0xFF0000 ],
[ 0x0000B1, 0x0000E2, 0x000074, 0x0000D8, 0x00B100, 0x00E200,
  0x007400, 0x00D800, 0xB10000, 0xE20000, 0x740000, 0xD80000 ]];

desd24:=
   [LinearCode<GF(2),24|[Prune(Intseq(n+0x1000000,2)) : n in code]>
     : code in desd24genmats];/******
\end{lstlisting}
The function \verb+subcodes+ takes a doubly even code
$C$ containing $R$ as an argument, and returns
the list of subcodes of codimension $1$ of $C$ satisfying
$C\supset R$ up to the action of $\Aut(C)$.
\begin{lstlisting}[firstline=2,name=prog]
******/
subcodes:=function(C,R)
   A:=AutomorphismGroup(C);
   P:=PermutationModule(A,GF(2));
   DC:=Dual(C);
   DR:=Dual(R);
   PDC:=sub<P | VectorSpace(DC)>;
   PDR,e:=sub<P | VectorSpace(DR)>;
   M,p:=quo<PDR | PDC>;
   G:=MatrixGroup(M);
   X:=[DR| o[1] @@ p @ e : o in Orbits(G) | not 0 in o];
   overcodes:=[sub<DR|DC,x> : x in X];
   return [Dual(CC) : CC in overcodes];
end function;/******
\end{lstlisting}
Given a sequence of pairs of a code and a number,
the function \verb+uptoequivalenceDE+
returns a subsequence of complete representatives of codes up to equivalence
with the largest numbers appearing in the second components.
\begin{lstlisting}[firstline=2,name=prog]
******/
uptoequivalenceDE:=function(Ds)
   Css:=[];
   for D in Ds do
      ord := #AutomorphismGroup(D[1]);
      we  := WeightEnumerator(D[1]);
      if not exists(v){i:i in [1..#Css]|
           Css[i][1] eq ord and
           Css[i][2] eq we and
           IsEquivalent(Css[i][3], D[1])} then
         Append(~Css, <ord, we, D[1], D[2]>);
      else
         Css[v][4]:=Max([Css[v][4], D[2]]);
      end if;
   end for;
   return [<D[3],D[4]>: D in Css ];
end function;/******
\end{lstlisting}
\subsection*{Basic operation for codes}
Given a pair of vectors, the functions \verb+entrywiseProduct+ and \verb+CstarC+
return $c_1\ast c_2 = c_1\cap c_2 $ as the support
and $C\ast C =\langle c_1\ast c_2 \mid c_1, c_2\in C \rangle$ respectively.
\begin{lstlisting}[firstline=2,name=prog]
******/
entrywiseProduct:=func<x,y|
   CharacteristicVector(Parent(x), Support(x) meet Support(y))>;

CstarC:=function(D)
   k:=Dimension(D);
   CC:=LinearCode<GF(2), Length(D)|
      [entrywiseProduct(D.i,D.j):i,j in [1..k] | i lt j] cat
      [D.i : i in [1..k]]>;
   return CC;
end function;/******
\end{lstlisting}
Given codewords $x, y$ of a doubly even code $C$, the functions \verb+QForm+
and \verb+BForm+ return $Q(x)$ and $B(x,y)$ respectively.
\begin{lstlisting}[firstline=2,name=prog]
******/
QForm:=func<u|GF(2)!(Weight(u) div 4)>;
BForm:=func<u,v|GF(2)!(#(Support(u) meet Support(v)) div 2)>;/******
\end{lstlisting}
Given a vector $x$ and a doubly even code $D$ with a basis $\{u_1, u_2,\ldots, u_k\}$,
the function \verb+BFormArray+ returns an array $(B(x,u_i))_{i}$.
Given doubly even codes $C, D$ with respective bases $\{u_1, u_2,\ldots, u_k\}$
and $\{v_1, v_2,\ldots, v_l\}$, the function \verb+BFormMatrix+
returns a matrix $(B(u_i,v_i))_{i,j}$.
\begin{lstlisting}[firstline=2,name=prog]
******/
BFormArray:=function(x,D)
   kD:=Dimension(D);
   return [BForm(x,D.j) : j in [1..kD]];
end function;

BFormMatrix:=function(C,D)
   kC:=Dimension(C);
   kD:=Dimension(D);
   M:=Matrix(GF(2), kC, kD,
               [BFormArray(C.i, D) : i in [1..kC]]);
   return M;
end function;/******
\end{lstlisting}
Given a doubly even code $C$, the functions \verb+Cmeetrad+ and
\verb+CmeetRad+
return the subcode $C\cap \rad C$ and $C\cap \Rad C$ respectively,
applying Lemma~\ref{lem:RL}.
\begin{lstlisting}[firstline=2,name=prog]
******/
Cmeetrad:=function(C)
   D:=Dual(CstarC(C)) meet C;
   H:=VectorSpace(GF(2),Dimension(C));
   VD:=VectorSpace(D);
   g:=hom<VD->H|BFormMatrix(D,C)>;
   rad:=sub<D|Kernel(g)>;
   return rad;
end function;

CmeetRad:=function(C)
   rad:=Cmeetrad(C);
   k:=Dimension(rad);
   H:=VectorSpace(GF(2),1);
   VD:=VectorSpace(rad);
   g:=hom<VD->H| [[QForm(rad.i)]:i in [1..k]]>;
   Rad:=sub<rad|Kernel(g)>;
   return Rad;
end function;/******
\end{lstlisting}
\subsection*{Doubly even codes which contain each triply even radical}
The function \verb+outsideVectors+ returns a complete list of representatives
of cosets $(C\ast C)^\perp/((C\ast C)^\perp\cap C)$ up to the action of
$\Aut(C)$.
Given a doubly even code $C$, the function \verb+existsOutsideRad+
returns true if and only if $\Rad C\not\subset C$,
applying Lemma~\ref{lem:outrad} and Lemma~\ref{lem:outRad}.
\begin{lstlisting}[firstline=2,name=prog]
******/
outsideVectors:=function(C)
   U:=Generic(C);
   A:=AutomorphismGroup(C);
   P:=PermutationModule(A,GF(2));
   D:=Dual(CstarC(C));
   E:=C meet D;
   PD,e:=sub<P | VectorSpace(D)>;
   PDC:=sub<P | VectorSpace(E)>;
   M,p:=quo<PD | PDC>;
   G:=MatrixGroup(M);
   return {D!(o[1] @@ p @ e) : o in Orbits(G) | not 0 in o};
end function;

existsOutsideRad:=function(C)
   H:=VectorSpace(GF(2),Dimension(C));
   D:=Dual(CstarC(C)) meet C;
   VD:=VectorSpace(D);
   g:=hom<VD->H| BFormMatrix(VD,C)>;
   Im:=Image(g);
   rad := Kernel(g);
   b1:=exists(u){ i : i in [1..Dimension(rad)] | QForm(rad.i) ne 0};
   X:=outsideVectors(C);
   b2:=exists(v){x : x in X |
      imgx in Im and (b1 or QForm(x+imgx @@ g) eq 0)
      where imgx:= H!BFormArray(x,C)};
   return b2;
end function;/******
\end{lstlisting}
The record \verb+RF+ equips the following objects for
a doubly even code of length $24$.
\begin{lstlisting}[firstline=2,name=prog]
******/
RF:=recformat<
   C,    // the original code
   R,    // the triply even radical of C
   prd,  // the max dim of radical of supcode of codim = 1
   CR,   // the quotient space C/R
   p,    // the projection C -> C/R
   X,    // the array [ x in CR | Q(x) = 0 ]
   px,   // the projection V(X)->C/R
   CC,   // the triply even check code
   AutCR // Aut(C) meet Aut(R)
>;/******
\end{lstlisting}
Given a doubly even code $C$ and its triply even radical $R$,
the procedure \verb+profiles+ constructs
the quotient $C/R$, the projection $p: C\to C/R$,
the singular points $X$, the automorphism group $\Aut(C)\cap\Aut(R)$
and the triply even check code, and then
returns a record containing them.
\begin{lstlisting}[firstline=2,name=prog]
******/
profiles:=function(C, prd)
   s:=rec<RF | C:=C, prd:=prd>;
   s`R:=CmeetRad(C);
   s`CR,s`p:=VectorSpace(C)/VectorSpace(s`R);
   s`X:=[x:x in s`CR|QForm(x @@ s`p) eq 0];
   M:=Matrix(GF(2), #s`X, Dimension(s`CR), s`X);
   s`px:=hom<VectorSpace(GF(2),#s`X)->s`CR|M>;
   s`AutCR:=AutomorphismGroup(C)
             meet AutomorphismGroup(s`R);
   s`CC:=LinearCode(Kernel(s`px));
   return s;
end function;/******
\end{lstlisting}
The procedure \verb+constAllSubcodeContainsRad+ constructs the list of all
doubly even codes of length $24$ containing its triply even radical.
\begin{lstlisting}[firstline=2,name=prog]
******/
constAllSubcodeContainsRad:=function(maxcodes24)
   codes:=[[ profiles(D, 0) : D in maxcodes24]];
   print "=> Now, constructing all admissible doubly even codes of length 24...";
   for i in [1..9] do
      d:=12-i;
      ovcodes:=codes[#codes];
      reps:=&cat[[<C,Dimension(s`R)>:C in subcodes(s`C, s`R)]: s in ovcodes];
      reps:=uptoequivalenceDE(reps);
      reps:=[S : S in reps | not existsOutsideRad(S[1])];
      printf "=> Completed for dim=%3o, the number of codes=%4o.\n", d, #reps;
      Append(~codes, [profiles(D[1], D[2]) :D in reps]);
      if IsEmpty(reps) then
            break i;
      end if;
   end for;
   printf "=> This is the expected result : %o.\n",
           [#x:x in codes] eq [9,42,160,377,437,220,36,1, 0];
   return &cat(codes);
end function;/******
\end{lstlisting}
\subsection*{Identification of maximal triply even codes}
Given a triply even code, the function \verb+isMaximal+ returns true
if and only if the code is a maximal triply even code.
\begin{lstlisting}[firstline=2,name=prog]
******/
isMaximal:=function(C)
   D:=CodeComplement(Dual(CstarC(C)), C);
   t:=exists(u){x:x in D| x ne 0 and QForm(x) eq 0
      and forall(v){i:i in [1..Dimension(C)] |  BForm(x,C.i) eq 0}};
   return not t;
end function;/******
\end{lstlisting}
The procedure \verb+appendCode+ appends a new maximal
triply even code to the list of codes.
\begin{lstlisting}[firstline=2,name=prog]
******/
appendCode:=procedure(~codenum, ~maxcodes, reps, Ds, id)
   codenum:=codenum + #Ds;
   for D in Ds do
      if isMaximal(D) then
         Append(~maxcodes, D);
         invt:=<Dimension(D),NumberOfWords(D,8)>;
         id0:=Position(reps[2],invt);
         if id0 ne 0 and not IsEquivalent(D,reps[1][id0]) then
             id0 := 0;
         end if;
         printf
          "Found a MTE code = Rep.%2o of dim=%2o from DE code No.%o : %o.\n",
          id0, Dimension(D), id, reps[3][id0+1];
      end if;
   end for;
end procedure;/******
\end{lstlisting}

\subsection*{Triply even codes constructed from the combinations with $C_1=C_2$}
We enumerate all codes obtained from the method in Proposition~\ref{prop:D}
with $C_1=C_2$.

Given a doubly even code and its triply even radical,
the function \verb+constDoubleCosetsCC+ returns
the representatives of double cosets
\begin{equation}\label{eq:dcoset}
\cG_0(C,R)\backslash \cG_1(C,R)/\cG_0(C,R).
\end{equation}
\begin{lstlisting}[firstline=2,name=prog]
******/
constDoubleCosetsCC:=function(s)
   CRs:={@x:x in s`CR|x ne 0@};
   SCRs:=Sym(CRs);
   G0:=sub<SCRs|{[((x @@ s`p)^g) @ s`p:x in CRs]:
                g in Generators(s`AutCR)}>;
   GLCR:=sub<SCRs|{[x^g:x in CRs]:
                g in Generators(GL(s`CR))}>;
   G1:=Stabilizer(GLCR, {x : x in s`X | x ne 0});
   return DoubleCosetRepresentatives(G1, G0, G0);
end function;/******
\end{lstlisting}
Given a doubly even code $C$
and the double cosets (\ref{eq:dcoset}),
the function \verb+resultingCodesCC+ returns triply even codes
constructed from the code $C$ using the method in
Proposition~\ref{prop:D} with $C_1=C_2=C$.
\begin{lstlisting}[firstline=2,name=prog]
******/
resultingCodesCC:=function(s, dc)
   D:=DirectSum(s`R, s`R);
   k:=Dimension(s`CR);
   M1:=Matrix([s`CR.i @@ s`p : i in [1..k]]);
   codes:=[D+LinearCode(HorizontalJoin(M1, M2)) where
     M2:=Matrix([((s`CR.i)^g) @@ s`p : i in [1..k]])
            : g in dc];
   return codes;
end function;/******
\end{lstlisting}
The object \verb+partsDB+ is the set of doubly even codes of length $24$ containing
its triply even radical.
The function \verb+duplextype+ returns the list of all maximal triply even codes and
the number of triply even codes of length $48$ constructed
from \verb+partsDB+ with $C_1=C_2$.
\begin{lstlisting}[firstline=2,name=prog]
******/
duplextype:=function(partsDB, reps)
   maxcodes:=[];
   codenum:=0;
   excodenum:=0;
   for id in [1..#partsDB] do
      s:=partsDB[id];
      k:=Dimension(s`CR);
      if k eq 0 then
         if Dimension(s`R) eq s`prd then
            excodenum:=excodenum+1;
         else
            D:=DirectSum(s`R, s`R);
            appendCode(~codenum, ~maxcodes, reps, [D], id);
         end if;
      else
         doubleCosets:=constDoubleCosetsCC(s);
         if Dimension(s`R) eq s`prd then
            Remove(~doubleCosets, 1);
            excodenum:=excodenum+1;
         end if;
         if not IsEmpty(doubleCosets) then
            list:=resultingCodesCC(s, doubleCosets);
            appendCode(~codenum, ~maxcodes, reps, list, id);
         end if;
      end if;
   end for;
   return maxcodes, codenum, excodenum;
end function;/******
\end{lstlisting}
\subsection*{Triply even codes constructed from the combinations with $C_1\not\cong C_2$}

We enumerate all codes obtained from the method in Proposition~\ref{prop:D}
with $C_1\not\cong C_2$.

Given a pair of doubly even codes and an isomorphism
between their triply even check codes,
the function \verb+isometry+ returns an isometry
between them.
\begin{lstlisting}[firstline=2,name=prog]
******/
isometry:=function(s1, s2, g)
   CX:=Image(s1`px);
   BCX, p:=quo<s1`CR|CX>;
   bCR1:=Basis(CX) cat [x @@ p : x in Basis(BCX)];
   bCX2:=[bCR1[i] @@ s1`px @ g @ s2`px : i in [1..Dimension(CX)]];
   BbCX2, p:=quo<s2`CR|bCX2>;
   bCR2:=bCX2 cat [x @@ p : x in Basis(BbCX2)];
   return  hom<s1`CR->s2`CR | [bCR1[i]->bCR2[i] : i in [1..#bCR1]]>;
end function;/******
\end{lstlisting}
Given the object \verb+partsDB+,
which is the set of doubly even codes
of length $24$ containing its triply even radical,
the function \verb+isometricPairsC1C2+ returns the list of isometric
pairs of distinct doubly even codes
and an isometry between them.
\begin{lstlisting}[firstline=2,name=prog]
******/
isometricPairsC1C2:=function(ss)
   C1C2s:=&cat[[<i, j, isometry(ss[i],ss[j],g)>
           : j in [i+1..#ss]
             | Dimension(ss[i]`CR) eq Dimension(ss[j]`CR)
               and #ss[i]`X eq #ss[j]`X and isEq
               where isEq, g := IsEquivalent(ss[i]`CC, ss[j]`CC)]
         : i in [1..#ss]];
   printf "The number of hybrid pairs = 125: %o.\n", #C1C2s eq 125;
   return C1C2s;
end function;/******
\end{lstlisting}
Given a pair of doubly even codes and an isometry,
the function \verb+constDoubleCosets+ returns the double cosets
\begin{equation}\label{eq:dcoseth}
h^{-1}\cG_0(C_2,R_2)h\backslash\cG_1(C_1,R_1)/\cG_0(C_1,R_1).
\end{equation}
\begin{lstlisting}[firstline=2,name=prog]
******/
constDoubleCosetsC1C2:=function(s1, s2, h)
   CRs:={@ x : x in s1`CR  | x ne 0 @};
   SCRs:=Sym(CRs);
   G01:=sub<SCRs | {[((x @@ s1`p)^g) @ s1`p
           : x in CRs] : g in Generators(s1`AutCR)}>;
   G02:=sub<SCRs | {[((x @ h @@ s2`p)^g) @ s2`p @@ h
           : x in CRs] : g in Generators(s2`AutCR)}>;
   GLCR:=sub<SCRs | {[x^g : x in CRs] : g in Generators(GL(s1`CR))}>;
   G1:=Stabilizer(GLCR, {x : x in s1`X | x ne 0});
   return DoubleCosetRepresentatives(G1, G01, G02);
end function;/******
\end{lstlisting}
Given a pair of doubly even codes $C_1$ and $C_2$,
an isometry $h$ from $C_1/R_1$ to $C_2/R_2$
and the double cosets (\ref{eq:dcoseth}),
the function \verb+resultingCodesC1C2+ returns triply even codes
constructed from the pair of codes using the method
in Proposition~\ref{prop:D}.
\begin{lstlisting}[firstline=2,name=prog]
******/
resultingCodesC1C2:=function(s1, s2, h, dc)
   k:=Dimension(s1`CR);
   D:=DirectSum(s1`R, s2`R);
   M1:=Matrix([s1`CR.i @@ s1`p : i in [1..k]]);
   codes:=[D+LinearCode(HorizontalJoin(M1, M2)) where
     M2:=Matrix([((s1`CR.i)^g) @ h @@ s2`p : i in [1..k]])
                                          : g in dc];
   return codes;
end function;/******
\end{lstlisting}
Recall that the object \verb+partsDB+ is the set of doubly even codes
of length $24$ containing its triply even radical.
The function \verb+hybridtype+ returns the list of all maximal triply even codes and
number of triply even codes of length $48$ constructed
from \verb+partsDB+ with $C_1\not\cong C_2$.
\begin{lstlisting}[firstline=2,name=prog]
******/
hybridtype:=function(partsDB, reps)
   maxcodes:=[];
   codenum:=0;
   c1c2s:=isometricPairsC1C2(partsDB);
   for id in c1c2s do
      s1:=partsDB[id[1]];
      s2:=partsDB[id[2]];
      h:=id[3];
      k:=Dimension(s1`CR);
      if k eq 0 then
         D:=DirectSum(s1`R, s2`R);
         appendCode(~codenum, ~maxcodes, reps, [D], <id[1],id[2]>);
      else
         doubleCosets:=constDoubleCosetsC1C2(s1, s2, h);
         list:=resultingCodesC1C2(s1, s2, h, doubleCosets);
         appendCode(~codenum, ~maxcodes, reps, list, <id[1], id[2]>);
      end if;
   end for;
   return maxcodes, codenum;
end function;/******
\end{lstlisting}
\subsection*{Representative examples of maximal triply even codes}
We give $10$ maximal triply even codes
$\{\teD(C) \mid C\in\Delta\}$ and $\ttgc{10}$
of length $48$.
These codes are constructed by the functions
\verb+tildeD+ and \verb+TriangularGraphCode+.
\begin{lstlisting}[firstline=2,name=prog]
******/
tildeD := function(C)
   R := CmeetRad(C);
   return Juxtaposition(C,C)+DirectSum(R,R);
end function;

TriangularGraph:=function(v)
   X:=SetToIndexedSet(Subsets({1..v},2));
   return #X, Matrix(GF(2), #X, #X, [[#(x meet y):y in X] : x in X]);
end function;

TriangularGraphCode:=function(v)
   n, M:=TriangularGraph(v);
   r:=(-n) mod 8;
   return PadCode(LinearCode(M),r) + RepetitionCode(GF(2), n+r);
end function;/******
\end{lstlisting}
The object \verb+repMTECodes+ is the list of $10$ maximal triply even codes
equipped with their dimensions and the numbers of their codewords of
weight $8$.
\begin{lstlisting}[firstline=2,name=prog]
******/
repMTECodes1:=[ tildeD(C) : C in desd24 ] cat [TriangularGraphCode(10)];
repMTECodes2:=[<Dimension(C),NumberOfWords(C,8)> : C in repMTECodes1];
repMTECodes3:=["New!", "tD( g_{24} )", "tD( d_{24}^{+} )",
               "tD( d_{12}^{2+} )", "tD( (d_{10}e_7^2)^{+} )",
               "tD( d_8^{3+} )", "tD( d_6^{4+} )", "tD( d_4^{6+} )",
               "tD( d_{16}^{+}\oplus e_8 )", "tD( e_8^{\oplus3}\} )",
               "tT_{10}"];
dim_repMTECodes:={* Dimension(C) : C in repMTECodes1*};
printf "Representative codes are inequivalent each other: %o.\n",
         #repMTECodes1 eq #Seqset(repMTECodes2) and
         dim_repMTECodes eq {* 9^^1, 13^^7, 14^^1, 15^^1 *};
repMTECodes:=<repMTECodes1,repMTECodes2,repMTECodes3>;/******
\end{lstlisting}

\subsection*{Non existence of the other maximal triply even code}
In this subsection, we aim to ensure
that there does not exist any maximal
triply even code of length $48$
except for the representative examples in the previous
subsection up to equivalence.

First, we enumerate all doubly even codes of length $24$
which contain their triply even radicals.
\begin{lstlisting}[firstline=2,name=prog]
******/
partsDB:=constAllSubcodeContainsRad(desd24);
table:=[[Integers()!0: j in [1..13-k]]:k in [1..9]];
for k in [1..#partsDB] do
    i:=13-Dimension(partsDB[k]`C);
    j:=Dimension(partsDB[k]`R);
    table[i][j]+:=1;
end for;
printf "The number of admissible codes is same as expected: %o.\n",
table eq
[
   [   7,  1,  1,  0, 0, 0, 0, 0, 0, 0, 0, 0 ],
   [  33,  6,  3,  0, 0, 0, 0, 0, 0, 0, 0 ],
   [ 130, 19, 10,  1, 0, 0, 0, 0, 0, 0 ],
   [ 308, 40, 23,  5, 0, 1, 0, 0, 0 ],
   [ 363, 37, 25, 10, 1, 1, 0, 0 ],
   [ 180, 16, 10, 11, 2, 1, 0 ],
   [  27,  2,  0,  4, 2, 1 ],
   [   0,  0,  0,  0, 1 ],
   [   0,  0,  0,  0 ]
];/******
\end{lstlisting}
Second, we check the maximality of triply even codes
constructed from all the doubly even codes in duplicate.
\begin{lstlisting}[firstline=2,name=prog]
******/
duplex_max, duplex_num, exduplex_num
        :=duplextype(partsDB, repMTECodes);
printf "%3o maximal codes of duplex type found.\n",#duplex_max;
printf "This is the expected result: %o.\n",
      <#duplex_max, duplex_num, exduplex_num> eq <30,214,1268>;/******
\end{lstlisting}
Next, we check the maximality of triply even codes
constructed from all the pairs of distinct doubly even codes.
\begin{lstlisting}[firstline=2,name=prog]
******/
hybrid_max, hybrid_num:=hybridtype(partsDB, repMTECodes);
printf "%3o maximal codes of hybrid type found.\n",#hybrid_max;
printf "This is the expected result: %o.\n",
      <#hybrid_max, hybrid_num> eq <5,225>;/******
\end{lstlisting}

\subsection*{Result}
A classification of triply even codes of length $48$ has been completed
into the $10$ codes.
This calculation has been completed
in the total time: 650.240 seconds, the total memory usage: 534.91MB
under the environment using
``Intel{\tiny \textregistered} Core{\tiny \texttrademark} 2 Duo CPU T7500 @ 2.20GHz''.

\end{document}